\documentclass[a4paper]{amsart}

\usepackage{amsmath}
\usepackage{amssymb}
\usepackage{amsthm}
\usepackage{enumitem}
\usepackage[T1]{fontenc}
\usepackage{mathtools}
\usepackage{microtype}
\usepackage[dvipsnames]{xcolor}
\usepackage[all,cmtip]{xy}

\usepackage{cancel}
\usepackage[pagebackref]{hyperref}
\usepackage[alphabetic]{amsrefs}
\usepackage[capitalise,noabbrev]{cleveref}

\hypersetup{
	colorlinks=true,
	linkcolor=MidnightBlue,
	citecolor=MidnightBlue,
	urlcolor=MidnightBlue
}

\usepackage[margin=1in]{geometry}
\theoremstyle{plain}
\newtheorem{theorem}{Theorem}[section]

\newtheorem{conjecture}[theorem]{Conjecture}
\newtheorem{corollary}[theorem]{Corollary}

\newtheorem{lemma}[theorem]{Lemma}
\newtheorem{proposition}[theorem]{Proposition}
\newtheorem{question}[theorem]{Question}

\theoremstyle{definition}

\newtheorem{example}[theorem]{Example}
\newtheorem*{ack}{Acknowledgement}

\theoremstyle{remark}

\newtheorem{remark}[theorem]{Remark}

\newtheorem*{remark*}{Remark}
\newtheorem*{notation*}{Notation and Terminology}

\numberwithin{equation}{section}

\DeclareMathOperator{\Aut}{Aut}
\DeclareMathOperator{\Bim}{Bim}

\DeclareMathOperator{\GL}{GL}
\DeclareMathOperator{\Gal}{Gal}

\DeclareMathOperator{\SL}{SL}

\DeclareMathOperator{\Hom}{Hom}

\DeclareMathOperator{\id}{id}
\DeclareMathOperator{\Alb}{Alb}
\DeclareMathOperator{\alb}{alb}

\DeclareMathOperator{\End}{End}

\DeclareMathOperator{\Pic}{Pic}

\DeclareMathOperator{\Sing}{Sing}
\newcommand{\CC}{\mathbb{C}}
\newcommand{\OO}{\mathcal{O}}
\newcommand{\PP}{\mathbb{P}}
\newcommand{\QQ}{\mathbb{Q}}

\newcommand{\ZZ}{\mathbb{Z}}
\DeclareMathOperator{\Ker}{Ker}

\newcommand{\Romannum}[1]{\uppercase\expandafter{\romannumeral #1}}

\DeclarePairedDelimiter{\abs}{\lvert}{\rvert}

\DeclarePairedDelimiterX{\pair}[2]{\langle}{\rangle}{#1,#2}

\title[Wild automorphisms]{Wild automorphisms of compact complex spaces of lower dimensions}

\author{Jia Jia}
\address{National University of Singapore, Singapore 119076, Republic of Singapore}
\curraddr{Yau Mathematical Sciences Center, Jingzhai, Tsinghua University, Haidian District, Beijing, China, 100084}
\email{jia\_jia@u.nus.edu,mathjiajia@tsinghua.edu.cn}

\author{Long Wang}
\address{Shanghai Center for Mathematical Sciences, Fudan University, Jiangwan Campus, Shanghai, 200438, China;
and The University of Tokyo, 3-8-1 Komaba, Meguro-Ku, Tokyo 153-8914, Japan}
\email{wanglll@fudan.edu.cn}

\begin{document}

\begin{abstract}
    An automorphism of a compact complex space is called wild in the sense of Reichstein--Rogalski--Zhang
    if there is no non-trivial proper invariant analytic subset.
    We show that a compact complex surface admitting a wild automorphism
    is either a complex torus or an Inoue surface of certain type,
    and this wild automorphism has zero entropy.
    As a by-product of our argument,
    we obtain new results about the automorphism groups of Inoue surfaces.
    We also study wild automorphisms of compact K\"ahler threefolds or fourfolds,
    and generalise the results of Oguiso--Zhang from the projective case to the K\"ahler case.
\end{abstract}

\maketitle
\setcounter{tocdepth}{1}
%\tableofcontents

\section{Introduction}\label{sec:introduction}

Let \(X\) be a compact complex space.
We will use the \emph{analytic Zariski topology} on \(X\)
whose closed sets are all analytic sets (cf.~\cite{grauert1984coherent}*{Page~211}).
An automorphism \(\sigma \in \Aut(X)\) is called \emph{wild}
in the sense of Reichstein--Rogalski--Zhang (\cite{reichstein2006projectively})
if for any non-empty analytic subset \(Z\) of \(X\) satisfying \(\sigma(Z) = Z\), we have \(Z=X\);
or equivalently, for every point \(x \in X\), its orbit \(\{\sigma^n(x) \mid n\geq 0\}\) is Zariski dense in \(X\).

The following two conjectures generalise
\cite{reichstein2006projectively}*{Conjecture~0.3} and \cite{oguiso2022wild}*{Conjecture~1.4}
from the projective case to the K\"ahler case.

\begin{conjecture}[cf.~\cite{reichstein2006projectively}*{Conjecture~0.3}]\label{conj:wild}
    Assume that a compact K\"ahler space \(X\) admits a wild automorphism.
    Then \(X\) is isomorphic to a complex torus.
\end{conjecture}

\begin{conjecture}[cf.~\cite{oguiso2022wild}*{Conjecture~1.4}]\label{conj:zero_entropy}
    Every wild automorphism \(\sigma\) of a compact K\"ahler space \(X\) has zero entropy.
\end{conjecture}

When \(X\) is a projective variety,
wild automorphisms are related with the twisted homogeneous coordinate rings,
which play a role in noncommutative algebraic geometry (see~\cite{reichstein2006projectively}).
The study of wild automorphisms is also of interest from the viewpoint of dynamical systems
(see~\cite{cantat2020automorphisms}). 

For a compact K\"ahler surface \(X\) with a wild automorphism \(\sigma\), it is well-known that \(X\) is a complex torus, and \(\sigma\) is of the certain form, a priori, of zero entropy (see~\cite{reichstein2006projectively}*{Theorem~6.5} and \cite{cantat2020automorphisms}*{Theorem~6.10}).

In this article, we consider compact complex surfaces (not necessarily K\"ahler) with a wild automorphism.
We give a characterisation of such surfaces and show that there do exist examples of non-K\"ahler surfaces
that admit a wild automorphism.

\begin{theorem}\label{thm:wild_dim2}
    Let \(X\) be a compact complex space of dimension \(\leq 2\).
    Assume that \(X\) admits a wild automorphism \(\sigma\).
    Then we have:
    \begin{enumerate}
        \item \(X\) is either a complex torus or an Inoue surface of type \(S_M^{(+)}\),
              and \(\sigma\) has zero entropy.
        \item Both cases in (1) occur:
              there are pairs \((X', \sigma')\)
              where \(X'\) is a complex torus or an Inoue surface of type \(S_M^{(+)}\)
              and \(\sigma'\) acts on \(X'\) as a wild automorphism.
    \end{enumerate}
\end{theorem}

We refer to \cref{sec:inoue_surfaces} for the definition and the constructions of Inoue surfaces.
In \cref{sec:inoue_surfaces},
we will construct examples of wild automorphisms of Inoue surfaces of type \(S_M^{(+)}\) (see Example \ref{ex:smp}).
We remark that there are examples of wild automorphisms of complex abelian surfaces.
More strongly, there are complex abelian surfaces with an automorphism of which all orbits are Euclidean dense
(see~\cite{cantat2020automorphisms}*{Example~6.6 and Lemma~6.7}).

\begin{question}[cf.~\cite{cantat2022free}*{Section~4.2}]
    Are there Inoue surfaces of type \(S_M^{(+)}\) with an automorphism of which all orbits are Euclidean dense?
\end{question}

As a by-product of our argument,
we obtain new results about the automorphism groups of Inoue surfaces,
which might be of independent interest.
\Cref{thm:aut_inoue} below gives a more refined structure than \cite{jia2022automorphisms}*{Section~6}.
We remark that Inoue surfaces are divided into three different types:
\(S_M\), \(S_M^{(+)}\) and \(S_M^{(-)}\).

\begin{theorem}[=~\cref{thm:aut_sm,thm:aut_smp,thm:aut_smn}]
\label{thm:aut_inoue}
    Let \(X\) be an Inoue surface.
    \begin{enumerate}[wide=0pt, leftmargin=*]
        \item If \(X\) is either of type \(S_M\) or \(S_M^{(-)}\),
              then the (biholomorphic) automorphism group \(\Aut(X)\) is finite.
        \item If \(X\) is of type \(S_M^{(+)}\),
              the neutral connected component \(\Aut_0(X) \simeq \mathbb{C}^*\) and \(\Aut(X)/\Aut_0(X)\) is finite.
    \end{enumerate}
\end{theorem}

Fujiki \cite{fujiki2009automorphisms} has studied the automorphism groups of parabolic Inoue surfaces.
% Note that a parabolic Inoue surface has positive second Betti number, and hence is not an Inoue surface.
It is worth noting that a parabolic Inoue surface has positive second Betti number,
which distinguishes it from the usual Inoue surfaces.
% We also refer to \cite{cantat2022free}*{Section~4.2}.

We propose the following questions rather than conjectures due to the lack of evidence.

\begin{question}
    \leavevmode
    \begin{enumerate}[wide=0pt, leftmargin=*]
        \item Is a compact complex space in Fujiki's class \(\mathcal{C}\)
              admitting a wild automorphism a complex torus?
              % \item Is a compact complex space admitting a wild automorphism a solvmanifold?
        \item Does every wild automorphism of a compact complex space have zero entropy?
    \end{enumerate}
\end{question}

A compact complex space is called in \emph{Fujiki's class \(\mathcal{C}\)}
if it is bimeromorphic to a compact K\"ahler manifold.
In dimension two, a compact complex manifold is in Fujiki's class \(\mathcal{C}\) if and only if it is K\"ahler,
while starting from dimension three,
the category of Fujiki's class \(\mathcal{C}\) is strictly larger.
% A \emph{solvmanifold} is a compact homogeneous space of a connected solvable Lie group.
% Note that both complex tori and Inoue surfaces are solvmanifolds.
In particular, the answers to both two questions are affirmative in dimension two due to \cref{thm:wild_dim2}.

In the rest of this article,
we study Conjectures~\ref{conj:wild} and \ref{conj:zero_entropy} in dimension three and four.

\begin{theorem}\label{thm:main}
    Let \(X\) be a compact K\"{a}hler space of dimension three,
    and let \(\sigma\) be a wild automorphism of \(X\). Then
    %Then \(X\) is either a complex torus or a weak Calabi--Yau threefold, and \(\sigma\) has zero entropy.
    \begin{enumerate}[wide = 0pt, leftmargin=*]
        \item \(X\) is either a complex torus or a weak Calabi--Yau threefold;
        \item \(\sigma\) has zero entropy.
    \end{enumerate}
\end{theorem}

Here a smooth complex projective variety \(V\) is called
\begin{enumerate}
    \item a \emph{weak Calabi--Yau manifold}, if \(K_V \sim_{\QQ} 0\) and \(\pi_1(V)\) is finite;
    \item a \emph{Calabi--Yau manifold in the strict sense},
          if \(V\) is simply connected, \(K_V \sim 0\) and \(H^{j}(V,\mathcal{O}_{V}) = 0\) for \(0 < j < \dim V\).
\end{enumerate}

Let us remark that, \(X\) in \cref{thm:main} could not be a weak Calabi--Yau threefold
if one assumes the generalised non-vanishing conjecture
which predicts that any nef Cartier divisor on a Calabi--Yau threefold is effective
(\cite{oguiso2022wild}*{Theorem~7.4}, see also \cite{kirson2010wild}*{Theorem~4.7}).
The following proposition provides further evidence.

\begin{proposition}\label{prop:CY}
    Let \(X\) be a weak Calabi--Yau threefold, and let \(c_2(X)\) be the second Chern class of \(X\).
    Assume that either
    \begin{enumerate}
        \item \(c_2(X)\cdot D > 0\) for every non-torsion nef Cartier divisor \(D\) on \(X\); or
        \item there exists a non-torsion semi-ample Cartier divisor \(D\) on \(X\) such that \(c_2(X) \cdot D = 0\).
    \end{enumerate}
    Then \(X\) has no wild automorphism.
\end{proposition}

\begin{theorem}\label{thm:main_zero_entropy}
    \Cref{conj:zero_entropy} is true in all three cases below.
    \begin{enumerate}[wide=0pt, leftmargin=*]
        \item \(\dim X \leq 3\).
        \item \(\dim X = 4\), and the Kodaira dimension \(\kappa(X) \geq 0\),
        \item \(\dim X = 4\), and the irregularity \(q(X) \neq 1,2\).
    \end{enumerate}
\end{theorem}

\medskip
\noindent
\textbf{Difficulty in the non-algebraic setting.}
With the focus on Jordan property and so on,
both \cite{prokhorov2020automorphism}*{Section~7} and \cite{jia2022automorphisms}*{Section~6} studied the automorphism groups of Inoue surfaces.
We further determined the full automorphism group (\cref{thm:aut_inoue}),
which is one of the main ingredients of the proof of \cref{thm:wild_dim2}.
For this, we rely on a key \cref{lem:finite_quotient} which is of arithmetic flavor.

\Cref{thm:main,thm:main_zero_entropy} generalise the results of Oguiso--Zhang (\cite{oguiso2022wild}*{Theorems~1.2 and~1.5})
from the projective case to the K\"ahler case.
Compared with the projective case, the main difficulty here is to show the technical \cref{lem_oz2.11} and \cref{prop:lie_gp_action_contains_wild}.
In order to achieve this, we need some new ingredients which are analytic in nature.

% \Cref{thm:main,thm:main_zero_entropy} demonstrate that the generalisations
% (Conjectures~\ref{conj:wild} and \ref{conj:zero_entropy})
% from the projective case to the K\"ahler case are reasonable.
We are unable to generalise our results to the case of Fujiki's class \(\mathcal{C}\),
due to the lack of some fundamental tools,
such as the \(\Aut\)-equivariant K\"ahler model, the Beauville--Bogomolov decomposition and the Gromov--Yomdin theorem.
In authors' opinion,
it would be more interesting if there exist examples of non-K\"ahler manifolds
in Fujiki's class \(\mathcal{C}\) that admit a wild automorphism.

\subsection*{Structure of the paper}

In \cref{sec:preliminaries}, we collect some basic facts about wild automorphisms. In \cref{sec:tech_lem}, we prove a technical lemma for later use.
In \cref{sec:wild_dim2}, we study wild automorphisms of compact complex surfaces.
In \cref{sec:inoue_surfaces}, we study the automorphism groups of Inoue surfaces.
The proofs of \cref{thm:wild_dim2,thm:aut_inoue}
will be completed in \cref{sec:wild_dim2,sec:inoue_surfaces}.
In \cref{sec:wild_dim3}, we prove \cref{thm:main,prop:CY,thm:main_zero_entropy}.

\begin{ack}
    The authors would like to thank Professor Keiji Oguiso and Professor De-Qi Zhang
    for many valuable suggestions and inspiring discussions.
    The authors also thank the referee for comments.
    The second author thanks Department of Mathematics at National University of Singapore for warm hospitality.
    The first author is supported by a President's Graduate Scholarship from NUS.
    The second author is supported by JSPS KAKENHI Grant Number 21J10242.
\end{ack}

\section{Preliminaries}\label{sec:preliminaries}

We refer to \cite{barth2004compact} and to \cites{grauert1984coherent,ueno1975classification}
for the knowledge about compact complex surfaces and for general theory of compact complex spaces.
We also refer to \cites{oguiso2014aspects,dinh2017equidistribution}
for an overview of topological entropy and dynamical degrees.
When $f\colon X \to X$ is a surjective holomorphic map of a compact K\"ahler manifold $X$,
a corollary of the fundamental Gromov--Yomdin theorem says that
$f$ has zero entropy if and only if the first dynamical degree of $f$ is $1$
(see, e.g., \cite{oguiso2014aspects}*{Corollary~3.8~(4)}).

In what follows, we collect some results about wild automorphisms of compact complex spaces for later use. %The main sources are \cites{reichstein2006projectively,oguiso2022wild}.
We simply call a holomorphic map between complex spaces a morphism.

\subsection{Basic properties of wild automorphisms}

\begin{lemma}[cf.~\cite{oguiso2022wild}*{Lemma~2.5}]\label{lem:ppty_of_wild}
    Let \(X\) be a compact complex space and let \(\sigma\) be an automorphism on \(X\).
    \begin{enumerate}[wide = 0pt, leftmargin=*]
        \item If \(\sigma\) is wild then \(X\) is smooth.
        \item \(\sigma\) is wild if and only if \(\sigma^m\) is wild for some \(m\geq 1\)
              (and hence for all \(m\geq 1\)).
              In particular, a wild automorphism has infinite order.
    \end{enumerate}
\end{lemma}

\begin{proof}
    (1)
    Note that the singular locus \(\Sing X\) is an analytic subset of \(X\) and stabilised by every automorphism.
    Since \(\sigma\) acts as a wild automorphism on \(X\),
    \(\Sing X = \emptyset\) and \(X\) is smooth.

    For (2),
    if \(\sigma^m\) stabilised an analytic subset \(Z\) of \(X\),
    then \(\sigma\) stabilises the analytic subset \(\cup_{i=0}^{m-1} \sigma^i(Z)\) of \(X\).
\end{proof}

\begin{proposition}[cf.~\cite{oguiso2022wild}*{Proposition 2.6}]\label{prop:numerical_inv_of_wild}
    Let \(X\) be a compact complex space with a wild automorphism \(\sigma\).
    \begin{enumerate}[wide=0pt, leftmargin=*]
        \item The Euler--Poincar\'e characteristic \(\chi(\mathcal{O}_{X}) = 0\),
              and the topological Euler number \(e(X)=0\).
              In particular, \(X\) is not rationally connected.
        \item Let \(D\) be a Cartier divisor on \(X\) such that \(\sigma^* D \sim D\).
              Then either \(\abs{D} = \emptyset\) or \(D \sim 0\).
              In particular, the Kodaira dimension \(\kappa(X) \leq 0\);
              if \(\kappa(X) = 0\), then \(K_X \sim_{\mathbb{Q}} 0\).
        \item Suppose that \(X\) is K\"ahler and \(\kappa(X) = 0\).
              Then the Beauville--Bogomolov (minimal split) finite \'etale cover \(\widetilde{X}\) of \(X\)
              is a product of a complex torus \(T\)
              and some copies of Calabi--Yau manifolds \(C_i\) in the strict sense;
              % of odd dimension \(\geq 3\) and
              and a positive power of \(\sigma\) lifts to a diagonal action on \(\widetilde{X} = T \times \prod_i C_i\)
              whose action on each factor is wild.
    \end{enumerate}
\end{proposition}

\begin{proof}
    (1)
    Since both Lefschetz fixed point theorem and holomorphic Lefschetz fixed point theorem
    hold for compact complex manifolds,
    the arguments of \cite{reichstein2006projectively}*{Proposition~4.4, Remark~4.5} still hold.

    (2)
    Note that $X$ is smooth by \cref{lem:ppty_of_wild}~(1).
    So Cartier divisors and Weil divisors coincide on $X$ (see~\cite{ueno1975classification}*{Theorem~4.13}).
    Suppose that $\abs{D} \neq \emptyset$ and $D\not\sim 0$, then \(\abs{D} \cong \mathbb{P}^N\) with \(N \geq 1\).
    It follows that the action of \(\sigma\) on \(\abs{D}\) has a fixed point,
    which gives a \(\sigma\)-invariant proper analytic subset of $X$.
    This is a contradiction and proves the first claim.
    In particular, \(\kappa(X) \leq 0\).

    If \(\kappa(X) = 0\), then \(\abs{mK_X} \neq \emptyset\) for some \(m>0\)
    and there is a unique effective divisor \(E\) such that \(mK_X \sim E\).
    Then \(\sigma(E) = E\) implies that \(E=0\), which shows \(K_X \sim_{\mathbb{Q}} 0\).

    (3)
    The argument of \cite{oguiso2022wild}*{Proposition~2.6~(3)}
    using the Beauville--Bogomolov covering (cf.~\cite{beauville1983remarks}) still works.
\end{proof}

\subsection{Wild automorphisms on complex tori}

\begin{lemma}[cf.~\cite{reichstein2006projectively}*{Corollary 4.3}]\label{lem_rrz4.3}
    Let \(X\) be a complex torus,
    and suppose that \(Y \subseteq X\) is an irreducible subvariety of \(X\) such that \(Y\) admits a wild automorphism.
    Then \(Y\) is a translation of a subtorus of \(X\).
\end{lemma}

\begin{proof}
    We know that the Kodaira dimension \(\kappa(Y) \leq 0\) by Proposition \ref{prop:numerical_inv_of_wild} (2).
    Then a theorem of Ueno (\cite{ueno1975classification}*{Lemma~10.1 and Theorem~10.3})
    states that \(Y\) must be a translation of a subtorus of \(X\).
\end{proof}

Let \(X\) and \(Y\) be complex tori.
For any \(a \in X\) one has the \emph{translation} automorphism \(T_a \colon x \mapsto x + a\).
Let \(f\colon X \to Y\) be a morphism.
Then there is a unique (group) homomorphism \(\alpha \in \Hom(X,Y)\) such that
\(f = T_b \circ \alpha\), where \(b \in Y\).

\begin{theorem}[cf.~\cite{reichstein2006projectively}*{Theorem~7.2}]\label{thm_rrz7.2}
    Let \(X\) be a complex torus.
    Suppose that \(\sigma = T_b \circ \alpha\) is an automorphism of \(X\),
    where \(\alpha \in \End(X)\) is an automorphism and \(b \in X\).
    Let \(\beta = \alpha - \id\),
    and set \(S = \{b, \beta(b), \beta^2(b), \dots\} \subseteq X\).
    Then the following are equivalent:
    \begin{enumerate}[label=(\alph*)]
        \item \(\sigma\) is wild.
        \item \(\alpha\) is unipotent and \(S\) generates \(X\).
        \item \(\alpha\) is unipotent and the image \(\overline{b}\) of \(b\) generates \(X/\beta(X)\).
    \end{enumerate}
\end{theorem}

\begin{proof}
    The argument of \cite{reichstein2006projectively}*{Theorem~7.2} is still available under our setting.
    Notice that for \((b, c) \Rightarrow (a)\),
    one may use \cref{lem_rrz4.3} instead of \cite{reichstein2006projectively}*{Corollary~4.3}.
\end{proof}

\begin{proposition}[cf.~\cite{oguiso2022wild}*{Lemma~2.5 (6)}]\label{prop:entropy_tori}
    Let \(X\) be a complex torus and let \(\sigma\) be a wild automorphism on \(X\).
    Then \(\sigma\) has zero entropy.
\end{proposition}

\begin{proof}
    Let \(n = \dim X\).
    Note that \(H^j(X,\mathbb{C}) \simeq \wedge^j H^1(X,\mathbb{C})\) for \(1 \leq j \leq 2n\) since \(X\) is a complex torus.
    By \cref{thm_rrz7.2}, we can write \(\sigma = T_b \circ \alpha\) for some translation \(T_b\)
    and unipotent \(\alpha \in \End(X)\).
    Clearly \(T_b\) acts on \(H^1(X, \CC)\) and hence on each \(H^j(X, \CC)\) as an identity.
    We claim that the action of the unipotent \(\alpha \in \End(X)\) on \(H^1(X,\mathbb{C})\) is also unipotent (cf.~\cite{reichstein2006projectively}*{Theorem 8.5 (a)}).
    In fact, \(\End(X)_{\QQ} \coloneqq \End(X) \otimes \QQ\) is contained in \(M_{2n}(\QQ)\),
    the \(\QQ\)-vector space of matrices of rank \(2n\) with rational entries.
    So the claim follows from \cite{reichstein2006projectively}*{Lemma 8.2 (c)}.
    Then the actions of \(\alpha\) on \(H^{2k}(X,\mathbb{C})\) are all unipotent and hence the dynamical degrees \(d_k(\alpha)\) are all equal to \(1\) for \(1 \leq k \leq n\). Therefore, \(d_k(\sigma) = 1\) and \(\sigma\) has zero entropy (cf.~\cite{oguiso2014aspects}*{Corollary~3.8~(4)}).
    % By \cref{thm_rrz7.2}, we can write \(\sigma = T_b \circ \alpha\) for some translation \(T_b\) and unipotent \(\alpha \in \End(X)\). Then the action of the unipotent \(\alpha\) on \(H^1(X,\mathbb{C})\) is also unipotent (cf.~\cite{reichstein2006projectively}*{Theorem 8.5 (a)}). Note that \(H^j(X,\mathbb{C}) \simeq \wedge^j H^1(X,\mathbb{C})\) for \(1 \leq j \leq 2\dim X\) since \(X\) is a complex torus. Then the actions of \(\alpha\) on \(H^{2k}(X,\mathbb{C})\) are all unipotent and hence the dynamical degrees \(d_k(\alpha)\) are all equal to \(1\) for \(1 \leq k \leq \dim X\). Therefore, \(d_k(\sigma) = 1\) and \(\sigma\) has zero entropy (cf.~\cite{oguiso2014aspects}*{Corollary~3.8~(4)}).
\end{proof}

A compact complex space \(X\) is called a \emph{\textit{Q}-torus} if it has a
complex torus \(T_1\) as an \'etale finite cover; or equivalently, it is the quotient of a complex torus \(T_2\)
by a finite group acting freely on \(T_2\).

\begin{proposition}[cf.~\cite{oguiso2022wild}*{Proposition~2.13}]\label{prop:qtorus}
    Let \(X\) be a \textit{Q}-torus with a wild automorphism \(\sigma\).
    Then \(X\) is a complex torus.
\end{proposition}

\begin{proof}
    Let \(T \longrightarrow X\) be the minimal splitting cover of \(X\)
    in \cite{beauville1983remarks}*{Section 3}.
    The \(\sigma\) lifts to an automorphism on \(T\),
    also denoted as \(\sigma\).
    Note that the \(\sigma\) on \(T\) normalises \(H\coloneqq \Gal(T/X)\).
    Hence \(\sigma^{r!}\) centralises every element of \(H\),
    where \(r\coloneqq \abs{H}\).
    Since \(\sigma^{r!}\) is still wild by \cref{lem:ppty_of_wild} (2),
    \(H\) consists of translations by \cref{lem:centraliser_translation} below.
    Hence \(H = \{\id_T\}\) by the minimality of \(T\longrightarrow X\).
    Therefore, \(X=T\) and \(X\) is a complex torus.
\end{proof}

\begin{lemma}[cf.~\cite{kirson2010wild}*{Lemma~2.8}]\label{lem:centraliser_translation}
    Let \(\sigma\) be a wild automorphism of a complex torus \(T\).
    Let \(H \leq \Aut(T)\) be a finite subgroup centralised by \(\sigma\):
    \(\sigma h = h \sigma\) for all \(h \in H\).
    Then \(H\) consists of translations of \(T\).
\end{lemma}

\begin{proof}
    The argument of \cite{kirson2010wild}*{Lemma~2.8}
    is still valid for complex tori.
\end{proof}

\subsection{Wild automorphisms and fibrations}

\begin{lemma}[cf.~\cite{oguiso2022wild}*{Lemma~2.5}]\label{lem:ppty_of_wild2}
    Let \(X\) be a compact complex space and let \(\sigma\) be an automorphism on \(X\).
    \begin{enumerate}[wide = 0pt, leftmargin=*]
        \item Suppose that \(\sigma\) is wild and \(f\colon X\longrightarrow Y\)
              (resp.\ \(g\colon W\longrightarrow X\) with \(g(\Sing(W))\neq X\))
              is a \(\sigma\)-equivariant surjective morphism of compact complex spaces.
              Then \(f\) (resp.\ \(g\)) is a smooth morphism.
        \item Suppose that \(f\colon X\longrightarrow Y\) is a \(\sigma\)-equivariant
              surjective morphism to a compact complex space \(Y\).
              If the action \(\sigma\) on \(X\) is wild then so is the action of \(\sigma\) on \(Y\)
              (and hence \(Y\) is smooth).
        \item Suppose that \(f\colon X\longrightarrow Y\) is a \(\sigma\)-equivariant generically finite
              surjective morphism of compact complex spaces.
              Then the action of \(\sigma\) on \(X\) is wild if and only if so is the action of \(\sigma\) on \(Y\).
              Further, if this is the case,
              then \(f\colon X\longrightarrow Y\) is a finite \'etale morphism,
              and in particular, it is an isomorphism when \(f\) is bimeromorphic.
    \end{enumerate}
\end{lemma}

\begin{proof}
    (1)
    Let \(X_1\subseteq X\) be the subset consisting of points \(x \in X\) such that
    \(f\colon X\longrightarrow Y\) is not flat at \(x\).
    Then \(X_1\) is an analytic subset of \(X\) (cf.~\cite{frisch1967points}*{Th\'eor\`em~(\Romannum{6},~9)})
    and \(\sigma\)-stable.
    Hence \(X_1 = \emptyset\) and \(f\) is flat.
    Let \(X_2\) be the subset consisting of points \(x \in X\) such that \(X_{f(x)}\) is not smooth at \(x\).
    Then \(X_2\) is an analytic subset of \(X\) and \(\sigma\)-stable.
    % Page 108 of Several Complex Variables VII by Grauert, Peternell, and Remmert
    Hence \(X_2 = \emptyset\) and \(f\) is smooth.

    The case of \(g\) is similar.
    Consider the subsets of \(W\) over which \(g\) is non-flat or singular.
    Then each of them has an analytic image in \(X\) by Remmert's proper mapping theorem
    (see e.g., \cite{grauert1984coherent}*{Page~213}),
    and is different from \(X\) by \cite{frisch1967points}*{Proposition~(\Romannum{6},~14)}
    or our additional assumption that \(g(\Sing(W)) \neq X\).
    % Sard theorem Complex Analysis and Geometry pp. 57

    (2)
    Suppose the induced automorphism $\sigma_Y$ on $Y$ is not wild.
    Then there is a proper analytic subset $Y_1\subsetneq Y$ satisfying $\sigma_Y(Y_1) = Y_1$.
    So $f^{-1}(Y_1)\subsetneq X$ is a proper analytic subset satisfying $\sigma(f^{-1}(Y_1)) = f^{-1}(Y_1)$,
    which is a contradiction.

    (3) is similar to (1) and (2).
\end{proof}

\begin{lemma}[cf.~\cite{oguiso2022wild}*{Lemma 2.7}]\label{lem:oz2.7}
    Let \(X\) be a compact complex manifold with a wild automorphism \(\sigma\),
    let \(T\) be a complex torus and let \(f\colon X\longrightarrow T\) be a \(\sigma\)-equivariant
    morphism.
    Then the image \(Y\coloneqq f(X)\) is a subtorus of \(T\) and
    \(f\colon X\longrightarrow Y\) is a smooth surjective morphism.
    In particular, the Albanese map \(\alb_X\colon X\longrightarrow T=\Alb(X)\)
    is a surjective smooth morphism with connected fibres.
\end{lemma}

\begin{proof}
    By \cref{lem_rrz4.3}, after choosing a new origin for \(T\),
    the subvariety \(Y\) is a subtorus of \(T\).
    It follows from \cref{lem:ppty_of_wild2}~(1) that the surjective morphism \(X \to Y = f(X)\) is smooth.

    For the last assertion,
    note that \(\alb_X\colon X\longrightarrow \Alb(X)\) is \(\Aut(X)\)- and hence \(\sigma\)-equivariant,
    and \(\Alb(X)\) is generated by the image \(\alb_X(X)\) (which is a complex torus),
    so \(\alb_X(X) = \Alb(X)\).
    Let \(g\colon X\longrightarrow W\) be the Stein factorisation of \(\alb_X\).
    Since \(\alb_X\) is \(\sigma\)-equivariant, so is \(g\).
    By \cref{lem:ppty_of_wild2}, the induced morphisms \(\sigma_W\) and \(\sigma_{\Alb(X)}\) are both wild,
    \(W\) is smooth,
    and the induced morphism \(X\longrightarrow W\) is smooth while \(W\longrightarrow \Alb(X)\) is \'etale.
    Hence \(W\) is also a complex torus.
    In fact, the surjective \'etale morphism \(W\longrightarrow \Alb(X)\) is an isomorphism
    by the universality of the Albanese morphism.
    This completes the proof of the Lemma.
\end{proof}

\begin{proposition}[{cf. \cite{oguiso2022wild}*{Proposition~2.15}}]\label{prop_oz2.15}
    Let \(X\) be a compact K\"{a}hler space of dimension \(n\) with a wild automorphism \(\sigma\).
    Then the Albanese morphism \(\alb_X\colon X \longrightarrow \Alb(X)\)
    is a smooth surjective morphism
    with connected fibres and hence \(0 \leq q(X) = \dim \Alb(X) \leq n\).
    Moreover, if \(q(X) = n\), then \(X = \Alb(X)\), a complex torus;
    if \(q(X) \geq n - 1\), then \(\sigma\) has zero entropy.
\end{proposition}

\begin{proof}
    The same proof of \cite{oguiso2022wild}*{Proposition~2.15} works under our setting.
    Notice that the product formula used there holds for compact K\"ahler manifolds (see
    \cite{dinh2012relative}*{Theorem 1.1}).
    Moreover, we use Lemma \ref{lem:oz2.7} and Proposition \ref{prop:entropy_tori}
    instead of \cite{oguiso2022wild}*{Lemmas~2.7 and 2.5~(6)}
\end{proof}

A \emph{maximal rational connected (MRC) fibration} on a uniruled compact K\"ahler manifold
has general fibres \(F\) rationally connected,
and the base is not uniruled (cf.~\cite{cao2020rational}*{Remark~6.10}).

\begin{lemma}[cf.~\cite{oguiso2022wild}*{Lemma 2.8}]\label{lem:mrc}
    Let \(X\) be a uniruled compact K\"ahler manifold of dimension \(\geq 1\),
    with a wild automorphism \(\sigma\).
    Then we can choose the maximal rationally connected (MRC) fibration \(X\longrightarrow Y\)
    to be a well-defined \(\sigma\)-equivariant surjective smooth morphism with \(0 < \dim Y < \dim X\).
    Further, the action of \(\sigma\) on \(Y\) is also wild.
\end{lemma}

\begin{proof}
    By \cite{zhong2021compact}*{Proof of Theorem~1.3 (1)},
    we can choose the MRC fibration \(X \dashrightarrow Y\) to be \(\sigma\)-equivariant.
    Moreover, the natural (surjective) bimeromorphic morphism \(\pi\colon \Gamma \to X\)
    from the graph of \(X\dashrightarrow Y\),
    is \(\sigma\)-equivariant,
    and hence \(\pi\) is an isomorphism by \cref{lem:ppty_of_wild2}~(3).
    Thus, we may assume that \(X=\Gamma \longrightarrow Y\) is a well-defined surjective holomorphic map.

    Since \(X\) is not rationally connected by \cref{prop:numerical_inv_of_wild}~(1),
    \(Y\) is not a point.
    Since \(X\) is uniruled, \(\dim Y < \dim X\).
    The action of \(\sigma\) on \(Y\) is wild by \cref{lem:ppty_of_wild2}~(2).
    This proves the lemma.
\end{proof}

\subsection{Auxiliary results}

In the rest of this section, we give two results from number theory,
which will be used in the study of the automorphism groups of Inoue surfaces.

\begin{lemma}\label{lem:elementary}
    Let $f(x) \coloneqq x^n + a_{n-1} x^{n-1} + \cdots + a_1 x + a_0 \in \ZZ[x]$ with $a_0 = \pm 1$.
    If $\alpha$ is a real root of $f(x)$,
    then either $\alpha$ is irrational, or $\alpha = \pm 1$.
\end{lemma}

\begin{proof}
    The proof is elementary.
    Assume that $\alpha = c/d$ is rational with $c, d$ coprime integers.
    Then
    \[
        c\cdot \left(c^{n-1} + a_{n-1} c^{n-2}d + \cdots + a_1 d^{n-1}\right) = - a_0 d^n.
    \]
    This implies $c = \pm 1$,
    since otherwise, a prime factor $p$ of $c$ satisfies $p \mid d$,
    which is a contradiction.
    A similar argument yields $d = \pm 1$.
    This shows the assertion.
\end{proof}

\begin{lemma}\label{lem:finite_quotient}
    Let \(M \in \GL_n(\mathbb{Z})\) be a diagonalisable matrix where \(n = 2\) or \(3\).
    Assume that \(M\) has either
    \begin{itemize}
        \item two real eigenvalues \(\alpha\) \((\neq \pm 1)\) and \(1/\alpha\) or \(-1/\alpha\), when \(n = 2\); or
        \item three eigenvalues \(\alpha\) \((\neq \pm 1)\), \(\beta\) and \(\overline{\beta}\) \((\beta \neq \overline{\beta})\), when \(n = 3\).
    \end{itemize}
    Denote
    \[
        \Gamma\coloneqq \{N\in \GL_n(\mathbb{Z}) \mid N \text{ and } M \text{ are simultaneously diagonalisable}\}.
    \]
    Then \(\Gamma \simeq U \times \mathbb{Z}\)
    where \(U\) is a finite group.
    In particular,
    if we denote by \(M^{\mathbb{Z}}\)
    the subgroup of \(\Gamma\) generated by \(M\),
    then the quotient \(\Gamma/M^{\mathbb{Z}}\) is finite.
\end{lemma}

\begin{proof}
    We adapt this proof from David E Speyer
    % \href{https://dept.math.lsa.umich.edu/~speyer/}{David E Speyer} 
    on MathOverflow\footnote{\url{https://mathoverflow.net/q/55646}}.

    It is known that the set of \(n\times n\) rational matrices which commute with \(M\)
    is a rational vector space spanned by \(\id, M, \dots, M^{n-1}\) over \(\mathbb{Q}\).
    Notice that the \(\mathbb{Q}\)-span of the powers of \(M\) is isomorphic,
    as a ring, to \(\mathbb{Q}[x]/p(x)\),
    where \(p(x)\) is the characteristic polynomial of \(M\).
    Since \(p(x)\) is irreducible over \(\mathbb{Q}\) (cf.~\cref{lem:elementary}),
    \(\mathbf{K} \coloneqq \mathbb{Q}[x]/p(x)\) is a number field,
    and the matrices in
    \(\GL_n(\mathbb{Q})\), which commute with \(M\) are isomorphic to \(\mathbf{K}^{*}\).
    The set of the matrices whose entries are in \(\mathbb{Z}\)
    forms an order \(\mathcal{O}\) in \(\mathbf{K}\).
    Those matrices of \(\mathcal{O}\) which are in \(\GL_n(\mathbb{Z})\),
    is the unit group \(\mathcal{O}^{*}\) of \(\mathcal{O}\).
    Then by Dirichlet's unit theorem, in each case, \(\mathcal{O}^{*} \simeq \mu(\mathcal{O}) \times \mathbb{Z}\),
    where \(\mu(\mathcal{O})\) is the finite cyclic group of roots of unity in \(\mathcal{O}\).
    % where \(r,s\) are the numbers of real and complex places of \(K\), respectively.
    % In our case, \(r=1\) and \(s=1\) and hence \(\mathcal{O}^{*}\) is
    % isomorphic to a finite group times \(\mathbb{Z}\).
    Since the subgroup \(M^{\mathbb{Z}} \leq \Gamma\) is isomorphic to \(\mathbb{Z}\),
    the quotient \(\Gamma/M^{\mathbb{Z}}\) is a finite group, which proves the last claim.
\end{proof}

\section{A technical lemma}\label{sec:tech_lem}

The aim of this section is to show the following technical lemma (\cref{lem_oz2.11}). This is a partial generalisation of \cite{oguiso2022wild}*{Lemma~2.11} from the projective case to the K\"ahler case, which is sufficient for our purpose.

\begin{lemma}\label{lem_oz2.11}
    Let \(X\) be a compact K\"ahler manifold with a wild automorphism \(\sigma\),
    let \(A\) be a complex torus and let \(f\colon X \longrightarrow A\) be a \(\sigma\)-equivariant
    surjective projective morphism with connected fibres of positive dimension.
    Assume general fibres of \(f\) are isomorphic to \(F\).
    Suppose that a positive power \(\sigma_A^s\) of \(\sigma_A\)
    fixes some big \((1,1)\)-class \(\alpha\) on \(A\) in \(H^{1,1}(A)\)
    (this holds if \(\dim A = 1\) or a positive power of \(\sigma_A\) is a translation on \(A\)).
    Then \(- K_F\) is not a big divisor.
\end{lemma}

Here we explain the claim in the bracket of the statement of \cref{lem_oz2.11}.
In fact, if \(\dim A = 1\), then \(\sigma_A^{12}\) is a translation on \(A\);
if a positive power \(\sigma_A^{s}\) of \(\sigma_A\) is a translation on \(A\),
then it fixes every big \((1,1)\)-class on \(A\).

Let us start with the following proposition which generalises \cite{reichstein2006projectively}*{Proposition~3.2}.

\begin{proposition}%[cf.~\cite{reichstein2006projectively}*{Proposition~3.2}]
\label{prop:lie_gp_action_contains_wild}
    Let \(X\) be a compact complex space in Fujiki's class \(\mathcal{C}\),
    and let \(G\leq \Aut(X)\) be a Lie group with finitely many components which acts biholomorphically on \(X\).
    Assume that there exists an element \(\sigma\in G\) acting on \(X\) as a wild automorphism.
    Then \(X\) is a complex torus and \(\sigma\) is a translation automorphism
    when a group structure in \(X\) is chosen properly.
\end{proposition}

\begin{proof}
    Recall that \(\Aut_0(X)\) is a complex Lie group.
    Since \(\sigma^m\) is wild by \cref{lem:ppty_of_wild}~(2),
    we may replace \(\sigma\) by \(\sigma^m\) and \(G\) by its neutral connected component \(G_0\),
    and so we may assume that \(G\leq \Aut_0(X)\).
    Let \(H\) be the (Zariski) closure of the subgroup \(\{\sigma^i\}_{i \in \mathbb{Z}}\) in \(G\).
    Then \(H\) is a closed Lie subgroup of \(G\).
    Since \(H\) is a closure of an abelian group, \(H\) itself is abelian.
    Without loss of generality,
    we may assume from now on that \(G=H\);
    in particular, \(G\) is abelian and closed.
    Note that now \(G\) is not assumed to have only finitely many components.

    By \cite{fujiki1978automorphism}*{Theorem~5.5},
    there is a short exact sequence
    \[
        1\longrightarrow L(X)\longrightarrow \Aut_0(X)\longrightarrow T(X)\longrightarrow 1
    \]
    where \(L(X)\) is a linear algebraic group and \(T(X)\) is a complex torus.
    The image of \(G\) in \(T(X)\) is closed and hence compact.
    In particular, the image of \(G\) in \(T(X)\) has finitely many components.
    Replacing \(\sigma\) by a power and replacing \(G\) accordingly,
    we may assume that the image of \(G\) in \(T(X)\) is a complex subtorus.
    Note that \(G \cap L(X)\) is an algebraic group and hence has finitely many components.
    It is not hard to see that \(G\) has finitely many components as well.
    After replacing \(\sigma\) by some power and replacing \(G\) accordingly,
    we may assume that \(G \leq \Aut_0(X)\) is irreducible.

    Assume that the neutral component \(U\) of the linear part of \(G\) is non-trivial.
    Then \(U\) being a commutative affine algebraic group,
    is isomorphic to \(\mathbb{G}_m^s \times \mathbb{G}_a^t\) for some \(s,t \geq 0\).
    Note that \(U\lhd G\) is a meromorphic group acting biholomorphically and meromorphically on \(X\)
    (see \cite{fujiki1978automorphism}*{Definition~2.1}).
    Let \(\pi\colon X \dashrightarrow Y = X/U\) be the quotient map to the cycle space
    (cf.~\cite{fujiki1978automorphism}*{Lemma~4.2}),
    so that \(G\) acts biholomorphically on \(Y\) and the map is \(G\)-equivariant.
    The set of points of indeterminacy \(S(\pi)\) is an analytic subset.
    Since \(S(\pi)\) is \(\sigma\)-invariant, one has \(S(\pi) = \emptyset\)
    and \(\pi\) is a holomorphic map.
    The general fibre of \(\pi\) is the closure of some orbit of \(U\)
    with respect to the natural action of \(U\) on \(X\).
    In particular, the (multi-)section at infinity is fixed by \(G\).
    But \(\sigma \in G\) is wild, a contradiction.
    Thus, the linear part of \(G\) is trivial and hence \(G\) is a complex torus.

    Choose any \(x \in X\).
    The rule \(g \mapsto gx\) defines a morphism \(f\colon G\longrightarrow X\), which is proper.
    By Remmert's proper mapping theorem (see e.g., \cite{grauert1984coherent}*{Page 213}),
    the image of \(f(G)\) is an analytic subset of \(X\).
    Since \(f(G)\) is \(\sigma\)-stable and \(\sigma\) is wild, \(f(G) = X\).

    Next, let \(G_0\) be the stabiliser of \(x\).
    Since \(G\) is abelian, \(G_0\) is the stabiliser of every point in \(Gx=X\).
    Since \(G\) is a subgroup of \(\Aut(X)\) and automorphisms of \(X\) are determined by
    their action on each point, the group \(G_0\) is trivial.
    Hence, the morphism \(f\colon G\longrightarrow X\) is bijective and hence biholomorphic.

    Then \(X\) is a complex torus.
    The isomorphism \(f\) transforms the translation automorphism \(T_{\sigma}\) of \(G\)
    to the translation automorphism \(T_{f(\sigma)} = \sigma\) of \(X\).
    Since we replaced \(\sigma\) by \(\sigma^n\) during the proof,
    we see that some power of the original \(\sigma\)
    is a translation automorphism of \(X\).
    Write \(\sigma = T_b \circ \alpha\) for some translation \(T_b\) and unipotent \(\alpha \in \End(X)\).
    Then \(\alpha^n = \id\) for some \(n \geq 1\).
    It follows that  \(\alpha\) is conjugate to  \(\alpha^n = \id\)
    (cf.~\cite{reichstein2006projectively}*{Remark~8.2~(a)}),
    so \(\alpha = \id\) and \(\sigma\) itself is a translation, as desired.
\end{proof}

\begin{proposition}[\cite{jia2022equivariant}*{Corollary~1.3}]\label{prop:finite_index}
    Let \(X\) be a compact complex space in Fujiki's class \(\mathcal{C}\),
    and we fix a big \((1,1)\)-class \([\alpha] \in H^{1,1}(X,\mathbb{R})\).
    Then the group
    \[
        \Aut_{[\alpha]}(X) \coloneqq \{g\in \Aut(X)\mid g^{*}[\alpha] = [\alpha]\}
    \]
    is a finite extension of the identity component \(\Aut_0(X)\) of \(\Aut(X)\),
    that is, \([\Aut_{[\alpha]}(X):\Aut_0(X)] <\infty\).
\end{proposition}

\begin{corollary}\label{cor:class_c_fixes_big_then_torus}
    Let \(X\) be a compact complex space in Fujiki's class \(\mathcal{C}\).
    Assume the \(X\) admits a wild automorphism \(\sigma\).
    Suppose that a positive power of \(\sigma\) fixes a big \((1,1)\)-class of \(X\).
    Then \(X\) is a complex torus.
\end{corollary}

\begin{proof}
    Assume that for some positive integer $s$,
    the wild automorphism \(\sigma^s\) fix a big \((1,1)\)-class \([\alpha] \in H^{1,1}(X,\mathbb{R})\).
    Then \(\sigma^s \in \Aut_{[\alpha]}(X)\),
    which is a Lie group with finitely many components by \cref{prop:finite_index}.
    Now we may apply \cref{prop:lie_gp_action_contains_wild} with \(G = \Aut_{[\alpha]}(X)\).
\end{proof}

Now we are ready to prove the technical lemma.

\begin{proof}[Proof of \cref{lem_oz2.11}]
    Note that the fibres of \(f\) are analytically isomorphic.
    In fact, let \(A_1\) be the set of points of \(A\) over which the fibre is not isomorphic to \(F\).
    Then \(A_1\) is \(\sigma\)-invariant and is disjoint with a non-empty open subset.
    Since \(\sigma_A\) is wild, \(A_1\) is empty.
    % otherwise, \(\sigma_A\) is not wild.
    By a theorem of Grauert--Fischer (see, e.g., \cite{barth2004compact}*{I.10.1}), \(f\) is locally trivial,
    i.e., \(f\) is a holomorphic fibre bundle.

    Suppose that \(- K_F = (- K_X)|_F\) is big.
    Then \(c_1(F)\) is a big \((1,1)\)-form on \(F\).
    Let \(\omega_F, \omega_A\) be K\"ahler forms on \(F\) and \(A\), respectively.
    Since both \(c_1(F)\) and \(\alpha\) are big (cf.~\cite{boucksom2002volume}*{Section 2.3}),
    there are plurisubharmonic functions \(\varphi, \psi\) on \(F\) and \(A\) respectively such that
    \begin{align*}
        c_1(F) + dd^c \varphi \geq \delta \omega_F \quad \text{for some} \quad \delta > 0, \\
        \alpha + dd^c \psi \geq \varepsilon \omega_A \quad \text{for some} \quad \varepsilon > 0.
    \end{align*}

    For each \(x \in X\), take an open neighbourhood \(U\) of \(a = f(x)\) in $A$ such that
    \(f\) is trivial on \(U\).
    On \(U\), let \(q\colon f^{-1}(U) \simeq F \times U \longrightarrow F\)
    be the first projection.
    Then \(c_1(X) - q^{*} c_1(F) = f^{*} \beta\) for some \((1,1)\)-class \(\beta\) on \(U\).
    Take \(n_U \gg 0\) such that \(\beta + n_U \varepsilon (\omega_A)|_U\) is a K\"ahler form on $U$.
    Then on \(f^{-1}(U)\),
    \begin{align*}
         & \phantom{\;=\;} c_1(X) + n_U f^{*}\alpha + dd^c \left( n_U \psi\circ f + \varphi \circ q \right)     \\
         & = c_1(X) + n_U f^{*}\alpha + n_U f^{*} dd^c \psi + q^{*} dd^c \varphi                                \\
         & = c_1(X) - q^{*} c_1(F) + q^{*} c_1(F) + q^* dd^c \varphi + n_U f^{*}\alpha + n_U f^{*} dd^c \varphi \\
         & = f^{*}\beta + n_U f^{*} (\alpha + dd^c \varphi) + q^{*} (c_1(F) + dd^c \varphi)                     \\
         & \geq f^{*}(\beta + n_U \varepsilon \omega_A) + \delta q^{*} \omega_F.
    \end{align*}
    Note that by the construction,
    \(f^{*}(\beta + n_U \varepsilon \omega_A) + \delta q^{*} \omega_F\)
    is a K\"ahler form on \(f^{-1}(U)\).
    It follows that \( \left(c_1(X) + n_U f^{*} (\alpha)\right)|_{f^{-1}(U)}\) is a big form on \(f^{-1}(U)\)
    as $n_U \psi \circ f + \varphi \circ q$ is a plurisubharmonic function on \(f^{-1}(U)\).
    Since \(A\) is compact, by taking a uniform \(n\),
    we can ensure that \(E \coloneqq c_1(X) + n f^{*} \alpha\) is a big form on \(X\).

    Now \(\sigma^s\) fixes the big class \([E]\).
    Thus, by \cref{cor:class_c_fixes_big_then_torus},
    \(X\) is a complex torus and hence \(K_F = K_X|_F = 0\), a contradiction.
\end{proof}

\section{Wild automorphisms in dimension two}\label{sec:wild_dim2}

In this section, we study wild automorphisms of compact complex surfaces and prove \cref{thm:wild_dim2}.
We distinguish two cases: K\"ahler (\cref{prop:wild_kahler_dim2})
and non-K\"ahler (\cref{prop:wild_nonkahler_dim2}).

The following result is well-known (see e.g., \cite{cantat2020automorphisms}*{Theorem~6.10}).
Here we give a proof for the sake of completeness.

\begin{proposition}\label{prop:wild_kahler_dim2}
    Let \(X\) be a compact K\"ahler space of dimension \(\leq 2\).
    Assume that \(X\) admits a wild automorphism \(\sigma\).
    Then \(X\) is a complex torus, and \(\sigma\) has zero entropy.
\end{proposition}

\begin{proof}
    Here we follow the argument of \cite{oguiso2022wild}*{Theorem 3.1}.
    Note that $X$ is smooth and $\kappa(X) \leq 0$
    by \cref{lem:ppty_of_wild}~(1) and \cref{prop:numerical_inv_of_wild}~(2).
    Thus, when $\dim X = 1$, $X$ is an elliptic curve by \cref{prop:numerical_inv_of_wild}~(1).
    Let us consider the case where $\dim X = 2$.

    If $\kappa(X) = -\infty$, by \cref{prop:numerical_inv_of_wild}~(1) and \cref{lem:mrc},
    $X$ admits a smooth fibration $f\colon X \to Y$
    with fibres $F$ smooth rational curve and $Y$ an elliptic curve.
    But then $F$ has ample $-K_F$,
    which contradicts with \cref{lem_oz2.11}.

    If $\kappa(X) = 0$,
    then $X$ is either a complex torus or a hyperelliptic surface by \cref{prop:numerical_inv_of_wild}~(2)
    and (1) together with the classification of minimal K\"ahler surfaces of Kodaira dimension $0$.
    Hence $X$ is a complex torus by \cref{prop:qtorus}.

    The second claim follows from the first one and \cref{prop:entropy_tori}.
\end{proof}

Now we deal with the non-K\"ahler case.

\begin{proposition}\label{prop:wild_nonkahler_dim2}
    Let \(X\) be a compact complex surface which is not K\"ahler.
    Suppose that \(X\) has a wild automorphism \(\sigma\).
    Then \(X\) is an Inoue surface of type \(S_M^{(+)}\), and \(\sigma\) has zero entropy.
\end{proposition}

\begin{proof}
    Notice that, by a result of Cantat (\cite{cantat1999dynamique}*{Proposition~1}),
    any automorphism of a non-K\"ahler surface has zero entropy.
    It remains to show that \(X\) is an Inoue surface of type \(S_M^{(+)}\).

    Note that any non-K\"ahler compact complex surface \(X\) has a unique minimal model \(X_{\mathrm{m}}\)
    and \(\Aut(X) \leq \Bim(X) \simeq \Bim(X_{\mathrm{m}}) = \Aut(X_{\mathrm{m}})\) (cf.~\cite{prokhorov2021automorphism}*{Proposition~3.5}), where \(\Bim(X)\) denotes the group of bimeromorphic transformations of \(X\).
    So there is a \(\sigma\)-equivariant surjective bimeromorphic morphism \(X \to X_{\mathrm{m}}\)
    and hence \(X \to X_{\mathrm{m}}\) is an isomorphism by \cref{lem:ppty_of_wild2}~(3),
    Therefore, \(X\) is minimal; see \cref{tab:non_kahler_minimal} below.

    \begin{table}[htbp]
        \caption{non-K\"ahler minimal smooth compact complex surfaces}
        \begin{tabular}{lccccc}
            \hline
            class of the surface \(X\)     & \(\kappa(X)\) & \(a(X)\) & \(b_1(X)\) & \(b_2(X)\) & \(e(X)\)   \\ \hline
            surfaces of class \Romannum{7} & \(-\infty\)   & \(0,1\)  & \(1\)      & \(\geq 0\) & \(\geq 0\) \\
            primary Kodaira surfaces       & \(0\)         & \(1\)    & \(3\)      & \(4\)      & \(0\)      \\
            secondary Kodaira surfaces     & \(0\)         & \(1\)    & \(1\)      & \(0\)      & \(0\)      \\
            properly elliptic surfaces     & \(1\)         & \(1\)    &            &            & \(\geq 0\) \\ \hline
        \end{tabular}
        \label{tab:non_kahler_minimal}
    \end{table}

    First we claim that the algebraic dimension \(a(X) = 0\).
    Suppose that the algebraic dimension \(a(X) = 1\).
    Let \(\pi\colon X\longrightarrow Y\) be the algebraic reduction with $Y$ a complex projective curve,
    which is \(\Aut(X)\)-equivariant (cf.~\cite{ueno1975classification}*{Chapter~\Romannum{1}, Definition~3.3}).
    In particular, \(\sigma\) descends to a wild automorphism \(\sigma_Y\) on the base \(Y\)
    (cf.~\cref{lem:ppty_of_wild2}~(2))
    and hence \(Y\) is an elliptic curve and \(\sigma_Y\) is of infinite order.
    As $X$ cannot be a Kodaira surface or a properly elliptic surface
    by \cite{prokhorov2020bounded}*{Lemma~2.4 and Proposition~1.2},
    we conclude that $\kappa(X) = -\infty$.
    So $X$ is a Hopf surface with $a(X) = 1$.
    But then the base $Y$ of the algebraic reduction of $X$ is $\PP^1$,
    % But then the base $Y \cong \PP^1$ by \cite{ueno1975classification}*{Chapter~\Romannum{4}, Theorem~11.5.2},
    which is a contradiction.
    Therefore, the algebraic dimension \(a(X) = 0\),
    a priori, $\kappa(X) = -\infty$.
    In particular, $X$ is a minimal surface of class \Romannum{7}.

    Now we claim that the second Betti number \(b_2(X) = 0\).
    Otherwise, the Euler number \(e(X) = b_2(X) > 0\) as $b_0(X) = b_1(X) = 1$,
    which is a contradiction to \cref{prop:numerical_inv_of_wild}~(1).
    Thus \(b_2(X) = 0\) and \(X\) is a minimal surface of class~\Romannum{7}.
    In other words, $X$ is either a Hopf surface with \(a(X) = 0\)
    or an Inoue surface (cf.~\cite{teleman1994projectively}*{Theorem 2.1}; see also \cite{bogomolov1976classification}).
    Note that there are finitely many (and at least one) curves on a Hopf surface with \(a(X) = 0\)
    (see \cite{barth2004compact}*{Theorems~\Romannum{4}~8.2 and \Romannum{5}~18.7}),
    which are \(\Aut(X)\)-invariant.
    Therefore, \(X\) cannot be a Hopf surface.
    This proves that \(X\) is an Inoue surface.

    Finally, by \cref{thm:aut_sm,thm:aut_smn} in the next section,
    we conclude that \(X\) must be an Inoue surface of type \(S_M^{(+)}\).
\end{proof}

\section{Automorphism groups of Inoue surfaces}\label{sec:inoue_surfaces}

In this section, we study the automorphism groups of Inoue surfaces.
Recall that an \emph{Inoue surface} \(X\) is a compact complex surface
obtained from \(W \coloneqq \mathbb{H}\times\mathbb{C}\) as a quotient
by an infinite discrete group, where \(\mathbb{H}\) is the upper half complex plane.
Inoue surfaces are minimal surfaces in class \Romannum{7},
contain no curve, and have the following numerical invariants:
\[
    a(X) = 0,\quad b_1(X) = 1,\quad b_2(X) = 0.
\]
There are three families of Inoue surfaces: \(S_M\), \(S_M^{(+)}\), and \(S_M^{(-)}\)
(cf.~\cite{inoue1974surfaces}), and we will study them separately.

\subsection{Type \texorpdfstring{\(S_M\)}{SM}}\label{sub:inoue_sm}

Let \(M = (m_{i,j})\in \SL_3(\mathbb{Z})\) be a matrix with eigenvalues
\(\alpha,\beta,\overline{\beta}\) such that \(\alpha>1\) and \(\beta\neq \overline{\beta}\).
Take \((a_1,a_2,a_3)^T\) to be a real eigenvector of \(M\) corresponding to \(\alpha\),
and \((b_1,b_2,b_3)^T\) an eigenvector corresponding to \(\beta\).
Let \(G_M\) be the group of automorphisms of \(W\) generated by
\begin{align*}
    g_0(w,z)   & = (\alpha w,\beta z),                    \\
    g_{i}(w,z) & = (w + a_{i}, z + b_{i}), \quad i=1,2,3,
\end{align*}
which satisfy these conditions
\begin{gather*}
    g_0 g_{i} g_0^{-1} = g_1^{m_{i,1}}g_2^{m_{i,2}}g_3^{m_{i,3}}, \\
    g_{i}g_{j} = g_{j}g_{i}, \quad i,j=1,2,3.
\end{gather*}
Note that \(G_M = G_1 \rtimes G_0\) where
\[ G_1 = \{g_1^{n_1}g_2^{n_2}g_3^{n_3} \mid n_i \in \mathbb{Z}\} \simeq \mathbb{Z}^3 \ \
    \text{and} \ \ G_0 = \langle g_0 \rangle \simeq \mathbb{Z}.\]
It can be shown that the action of \(G_M\) on \(W\) is free and properly discontinuous.
The quotient \(X\coloneqq W / G_M\) is an Inoue surface of type \(S_M\) (cf.~\cite{inoue1974surfaces}*{Section 2}).
Thus, there is a short exact sequence of groups
\[
    1\longrightarrow G_M\longrightarrow \widetilde{\Aut}(X)\longrightarrow \Aut(X)\longrightarrow 1
\]
where \(\widetilde{\Aut}(X)\) acts biholomorphically on \(W\),
which is the normaliser of \(G_M\) in \(\Aut(W)\).

As in \cite{jia2022automorphisms}*{Section 6.1},
\(\widetilde{\Aut}(X) = K \rtimes \Gamma\), where
\[
    K = \Big\{(w,z) \longmapsto
    \Big(w+\frac{1}{1-\alpha}\sum_{j=1}^{3}n_{i}a_{i},z+\frac{1}{1-\beta}\sum_{j=1}^{3}n_{i}b_{i}\Big)
    \mid n_i \in \mathbb{Z}\Big\} \simeq \mathbb{Z}^3
\]
and
\[
    \Gamma\coloneqq \{N\in \GL_3(\mathbb{Z}) \mid N \text{ and } M \text{ are simultaneously diagonalisable}\}.
\]
Now we have the following commutative diagram by the snake lemma
\begin{equation}\label{eq:sm_comm_diag}
    \xymatrix{
    {} & 1 \ar[d] & 1 \ar[d] & 1 \ar[d] & {} \\
    1 \ar[r] & G_1 \ar[r]^{\phi} \ar[d] & K \ar[r] \ar[d] & F \ar[r] \ar[d] & 1 \\
    1 \ar[r] & G_M \ar[r] \ar[d] & \widetilde{\Aut}(X) \ar[r] \ar[d] & \Aut(X) \ar[r] \ar[d] & 1 \\
    1 \ar[r] & G_0 \ar[r]^{\psi} \ar[d] & \Gamma \ar[r] \ar[d] & \Gamma/G_0 \ar[r] \ar[d] & 1 \\
    {} & 1 & 1 & 1 & {}
    }
\end{equation}
where \(\phi\) is given by
\[
    \phi(\mathbf{n})=\mathbf{n}\cdot (I-M) \quad \text{with } \mathbf{n}=(n_1,n_2,n_3)
\]
and \(\psi\) is defined via \(g_0 \mapsto M\).
Since \(I - M\) is invertible, the group \(F\) being the cokernel of \(\phi\), is finite.

Now by \cref{lem:finite_quotient}, \(\Gamma/G_0\) is finite.
Consequently, the diagram~\eqref{eq:sm_comm_diag} implies:

\begin{theorem}\label{thm:aut_sm}
    Let \(X\) be an Inoue surface of type \(S_M\).
    Then the automorphism group \(\Aut(X)\) is finite.
    In particular, \(X\) does not have any wild automorphism. \qed
\end{theorem}

\subsection{Type \texorpdfstring{\(S_M^{(+)}\)}{SM(+)}}\label{sub:inoue_smp}

Let \(M\in \SL_2(\mathbb{Z})\) be a matrix with two real eigenvalues \(\alpha\) and \(1/\alpha\) with \(\alpha>1\).
Let \((a_1, a_2)^T\) and \((b_1, b_2)^T\) be real eigenvectors of \(M\)
corresponding to \(\alpha\) and \(1/\alpha\), respectively,
and fix integers \(p_1,p_2,r\,(r\neq 0)\) and a complex number \(\tau\).
Define \((c_1, c_2)^T\) to be the solution of the following equation
\[
    (I-M)
    \begin{pmatrix}
        c_1 \\
        c_2
    \end{pmatrix}
    =
    \begin{pmatrix}
        e_1 \\
        e_2
    \end{pmatrix}
    +\frac{b_1a_2-b_2a_1}{r}
    \begin{pmatrix}
        p_1 \\
        p_2
    \end{pmatrix},
\]
where
\[
    e_{i}=\frac{1}{2}m_{i,1}(m_{i,1}-1)a_1b_1+\frac{1}{2}m_{i,2}(m_{i,2}-1)a_2b_2+m_{i,1}m_{i,2}b_1a_2, \quad i=1,2.
\]
Let \(G_M^{(+)}\) be the group of analytic automorphisms of \(W = \mathbb{H} \times \mathbb{C}\) generated by
\begin{align*}
    g_0\colon (w,z)   & \longmapsto (\alpha w, z+\tau),                     \\
    g_{i}\colon (w,z) & \longmapsto (w+a_{i}, z+b_{i}w+c_{i}), \quad i=1,2, \\
    g_3\colon (w,z)   & \longmapsto \Big(w, z+\frac{b_1a_2-b_2a_1}{r}\Big).
\end{align*}
We have the following relations between these generators
\begin{gather*}
    g_3g_{i}=g_{i}g_3 \quad \text{for } i = 0,1,2, \quad g_1^{-1}g_2^{-1}g_1g_2=g_3^r, \\
    g_0g_jg_0^{-1}=g_1^{m_{j,1}}g_2^{m_{j,2}}g_3^{p_{j}} \quad \text{for } j=1,2.
\end{gather*}
The action of \(G_M^{(+)}\) is free and properly discontinuous.
The quotient space \(X\coloneqq W/G_M^{(+)}\) is an Inoue surface of type \(S_M^{(+)}\)
(cf.~\cite{inoue1974surfaces}*{Section 3}).
Note that \(G_M^{(+)} \simeq H(r)\rtimes \mathbb{Z}\) as an abstract group,
where
\[
    H(r) = \big\langle g_1, g_2, g_3 \mid g_3g_{i}=g_{i}g_3, g_1^{-1}g_2^{-1}g_1g_2=g_3^r \big\rangle
\]
and \(\mathbb{Z}\) is generated by \(g_0\).
Further,
\begin{equation}\label{eq:smp_centre_one}
    Z(H(r)) = \langle g_3 \rangle \simeq \mathbb{Z} \ \ \text{and} \ \ H(r)/Z(H(r)) \simeq \mathbb{Z}^2,
\end{equation}
where $Z(H(r))$ is the centre of $H(r)$.

Similarly, there is a short exact sequence of groups
\[
    1\longrightarrow G_M^{(+)}\longrightarrow \widetilde{\Aut}(X)\longrightarrow \Aut(X)\longrightarrow 1,
\]
where \(\widetilde{\Aut}(X)\leq \Aut(W)\) is the normaliser of \(G_M^{(+)}\).

Let
\[
    \Gamma\coloneqq \{N\in \GL_2(\mathbb{Z}) \mid N \text{ and } M \text{ are simultaneously diagonalisable}\},
\]
which is an abelian group,
and let \(\tau\colon \widetilde{\Aut}(X)\longrightarrow \Gamma\) be the homomorphism
(not necessarily surjective) defined by \(u\mapsto N\),
where \(N\) is the matrix associated with \(u\)
as constructed in \cite{jia2022automorphisms}*{Section 6.2}.
Let \(K\) be the kernel of this homomorphism, with image \(\Gamma'\leq \Gamma\).
It is clear that any automorphism in \(K\) has the form
\[
    (w,z)\longmapsto (w+b,z+Aw+B)
\]
for some \(b\in \mathbb{R}\) and \(A,B\in \mathbb{C}\) satisfying some conditions
(\cite{jia2022automorphisms}*{(6.21) and (6.22)}).
An explicit calculation gives that the centre
\begin{equation}\label{eq:smp_centre_two}
    Z(K)\simeq \mathbb{C}
\end{equation}
which is generated by \((w,z)\longmapsto (w,z+B)\).
Further,
\begin{equation}\label{eq:smp_centre_three}
    K/Z(K)\simeq \mathbb{Z} \rtimes \mathbb{Z}
\end{equation}
which is generated by the images of
\[
    (w,z)\longmapsto (w+b,z) \quad \text{and} \quad (w,z)\longmapsto (w,z+Aw).
\]

Combining these, the snake lemma shows the diagrams below
are commutative with exact rows and columns.
\[
    \xymatrix@C=1pc{
    {} & 1 \ar[d] & 1 \ar[d] & 1 \ar[d] & {} \\
    1 \ar[r] & H(r) \ar[r] \ar[d] & K \ar[r] \ar[d] & F \ar[r] \ar[d] & 1 \\
    1 \ar[r] & G_M^{(+)} \ar[r] \ar[d] & \widetilde{\Aut}(X) \ar[r] \ar[d]^{\tau} & \Aut(X) \ar[r] \ar[d] & 1 \\
    1 \ar[r] & \mathbb{Z} \ar[r] \ar[d] & \Gamma' \ar[r] \ar[d] & \Gamma'/\mathbb{Z} \ar[r] \ar[d] & 1 \\
    {} & 1 & 1 & 1 & {}
    }
    \hspace{1em}
    \xymatrix@C=1pc{
    {} & 1 \ar[d] & 1 \ar[d] & 1 \ar[d] & {} \\
    1 \ar[r] & Z(H(r)) \ar[r] \ar[d] & H(r) \ar[r] \ar[d] & H(r)/Z(H(r)) \ar[r] \ar[d] & 1 \\
    1 \ar[r] & Z(K) \ar[r] \ar[d] & K \ar[r] \ar[d] &K/Z(K) \ar[r] \ar[d] & 1 \\
    1 \ar[r] & F' \ar[r] \ar[d] & F \ar[r] \ar[d] &F'' \ar[r] \ar[d] & 1 \\
    {} & 1 & 1 & 1 & {}
    }
\]
Note that \(F' \simeq \mathbb{C}^*\) by \cref{eq:smp_centre_one,eq:smp_centre_two},
and that \(F''\) is finite by \cref{eq:smp_centre_one,eq:smp_centre_three}.
Similarly to \cref{sub:inoue_sm},
\(\Gamma'/\mathbb{Z} \leq \Gamma/\mathbb{Z}\) is also a finite group by \cref{lem:finite_quotient}.
In conclusion,

\begin{theorem}\label{thm:aut_smp}
    Let \(X\) be an Inoue surface of type \(S_M^{(+)}\).
    Then the neutral connected component \(\Aut_0(X)\) of the automorphism group \(\Aut(X)\) is isomorphic to \(\mathbb{C}^*\)
    and the group of components \(\Aut(X)/\Aut_0(X)\) is finite.
    \qed
\end{theorem}

\begin{example}\label{ex:smp}
    Let \(X\) be an Inoue surface of type \(S_M^{(+)}\) defined by \(G_M^{(+)}\)
    and keep the notation in \cref{sub:inoue_smp}.
    It is known that, for \(n_1, n_2, l, k \in \mathbb{Z}\),
    \[
        g_1^{n_1} g_2^{n_2} g_3^{l} g_0^{k} (w,z)
        = \Big(\alpha^{k} w + \sum_j n_{j}a_j,
        z + k \tau + l\frac{b_1a_2 - b_2a_1}{r}
        + \Big(\sum_j n_{j}b_j\Big)\alpha^{k} w
        + \sum_j n_{j}c_j + e(n_1, n_2)\Big),
    \]
    where
    \[
        e(n_1, n_2)=\frac{1}{2}n_{1}(n_{1}-1)a_1b_1
        + \frac{1}{2}n_{2}(n_{2}-1)a_2b_2
        + n_{1}n_{2}a_2b_1.
    \]
    Pick an element \(B \in \mathbb{C} \simeq Z(K)\)
    such that \(\{nB \mid n \in \mathbb{Z}_{>0}\}\) is disjoint with the following set
    \[
        T \coloneqq \Big\{
        k \tau
        + l\frac{b_1a_2 - b_2a_1}{r}
        + \Big(\sum_j n_{j}b_j\Big) \Big(\sum_j n_{j}\alpha_j\Big) \frac{\alpha^k}{1 - \alpha^k}
        + \sum_j n_{j}c_j + e(n_1,n_2)
        \mid n_1, n_2, l, k \in \mathbb{Z}
        \Big\}.
    \]
    This is possible because only \(\tau\) in the definition of the set \(T\) could be a non-real number.
    Consider the automorphism
    \[
        \widetilde{\sigma} \in \widetilde{\Aut}(X), \quad (w, z) \longmapsto (w, z + B)
    \]
    and the corresponding automorphism \(\sigma \in \Aut(X)\).
    We claim that \(\sigma\) is a wild automorphism.

    It follows from the construction of \(\sigma\)
    that it does not have any periodic points.
    Indeed, if \((w,z) \in W\) corresponds to a \(\sigma\)-periodic point on \(X\),
    then \(\widetilde{\sigma}^n(w,z) = g_1^{n_1} g_2^{n_2} g_3^{l} g_0^{k} (w,z)\)
    for some \(n_1, n_2, l, k \in \mathbb{Z}\) and \(n \in \mathbb{Z}_{>0}\).
    Now we have
    \[
        \begin{cases}
            w = \alpha^{k} w + \sum_j n_{j}a_j, \\
            z + nB = z + k \tau + l\frac{b_1a_2 - b_2a_1}{r}
            + \Big(\sum_j n_{j}b_j\Big)\alpha^{k} w
            + \sum_j n_{j}c_j + e(n_1, n_2),
        \end{cases}
    \]
    which implies that \(nB \in T\), a contradiction.
    Since \(X\) does not contain any curve (\cite{inoue1974surfaces}*{Proposition~3.i) in page 277}),
    \(X\) cannot have any non-trivial proper \(\sigma\)-invariant analytic subset.
    Hence \(\sigma\) acts on \(X\) as a wild automorphism.
\end{example}

\subsection{Type \texorpdfstring{\(S_M^{(-)}\)}{SM(-)}}\label{sub:inoue_smn}

Let \(M\in \GL_2(\mathbb{Z})\) be a matrix with
two real eigenvalues \(\alpha\) and \(-1/\alpha\) with \(\alpha>1\).
Let \((a_1, a_2)^T\) and \((b_1, b_2)^T\) be real eigenvectors of \(M\)
corresponding to \(\alpha\) and \(1/\alpha\), respectively,
and fix integers \(p_1,p_2,r\,(r\neq 0)\) and a complex number \(\tau\).
Define \((c_1, c_2)^T\) to be the solution of the following equation
\[
    -(I+M)
    \begin{pmatrix}
        c_1 \\
        c_2
    \end{pmatrix}
    =
    \begin{pmatrix}
        e_1 \\
        e_2
    \end{pmatrix}
    +\frac{b_1a_2-b_2a_1}{r}
    \begin{pmatrix}
        p_1 \\
        p_2
    \end{pmatrix},
\]
where
\[
    e_{i} = \frac{1}{2}m_{i,1}(m_{i,1}-1)a_1b_1
    + \frac{1}{2}m_{i,2}(m_{i,2}-1)a_2b_2
    + m_{i,1}m_{i,2}b_1a_2,
    \quad i=1,2.
\]
Let \(G_M^{(-)}\) be the group of analytic automorphisms
of \(W = \mathbb{H} \times \mathbb{C}\) generated by
\begin{align*}
    g_0\colon (w,z)   & \longmapsto (\alpha w, -z),                         \\
    g_{i}\colon (w,z) & \longmapsto (w+a_{i}, z+b_{i}w+c_{i}), \quad i=1,2, \\
    g_3\colon (w,z)   & \longmapsto \Big(w, z+\frac{b_1a_2-b_2a_1}{r}\Big).
\end{align*}
We have the following relations between these generators
\begin{gather*}
    g_3g_{i}=g_{i}g_3 \quad \text{for } i = 1,2, \quad g_1^{-1}g_2^{-1}g_1g_2=g_3^r, \\
    g_0g_jg_0^{-1}=g_1^{m_{j,1}}g_2^{m_{j,2}}g_3^{p_{j}} \quad \text{for } j=1,2, \quad g_0 g_3 g_0^{-1} = g_3^{-1}.
\end{gather*}
The action of \(G_M^{(-)}\) is free and properly discontinuous.
The quotient space \(X\coloneqq W/G_M^{(-)}\) is an Inoue surface of type \(S_M^{(-)}\)
(cf.~\cite{inoue1974surfaces}*{Section 4}).
Note that \(G_M^{(-)} \simeq H(r)\rtimes \mathbb{Z}\) as an abstract group,
where \(H(r) = \langle g_1, g_2, g_3 \mid g_3g_{i}=g_{i}g_3, g_1^{-1}g_2^{-1}g_1g_2=g_3^r \rangle\)
and \(\mathbb{Z}\) is generated by \(g_0\).
%In fact, the centre
We have
\begin{equation}\label{eq:smn_centre_one}
    Z(H(r)) = \langle g_3 \rangle \simeq \mathbb{Z} \ \
    \text{and} \ \ H(r)/Z(H(r)) \simeq \mathbb{Z}^2.
\end{equation}

Similarly, there is a short exact sequence of groups
\[
    1\longrightarrow G_M^{(-)}\longrightarrow \widetilde{\Aut}(X)\longrightarrow \Aut(X)\longrightarrow 1,
\]
where \(\widetilde{\Aut}(X)\leq \Aut(W)\) is the normaliser of \(G_M^{(-)}\).

Let
\[
    \Gamma\coloneqq \{N\in \GL_2(\mathbb{Z}) \mid N \text{ and } M \text{ are simultaneously diagonalisable}\},
\]
which is an abelian group,
and let \(\tau\colon \widetilde{\Aut}(X)\longrightarrow \Gamma\) be the homomorphism
(not necessarily surjective) defined by \(u\mapsto N\),
where \(N\) is the matrix associated with \(u\)
as constructed in \cite{jia2022automorphisms}*{Section 6.2}.
Let \(K\) be the kernel of this homomorphism, with image \(\Gamma'\leq \Gamma\).
It is clear that any automorphism in \(K\) has the form
\[
    (w,z)\longmapsto (w+b,z+Aw+B)
\]
for some \(b\in \mathbb{R}\) and \(A,B\in \mathbb{C}\) satisfying some conditions.
An explicit calculation gives that the centre
\begin{equation}\label{eq:smn_centre_two}
    Z(K)\simeq \mathbb{Z}
\end{equation}
which is generated by
\[
    (w,z)\longmapsto \Big(w,z+\frac{1}{2} \cdot \frac{b_1a_2-b_2a_1}{r}\Big).
\]
Further,
\begin{equation}\label{eq:smn_centre_three}
    K/Z(K)\simeq \mathbb{Z} \rtimes \mathbb{Z}
\end{equation}
which is generated by the images of
\[
    (w,z)\longmapsto (w+b,z) \quad \text{and} \quad (w,z)\longmapsto (w,z+Aw).
\]

Combining these, the snake lemma shows the diagrams below
are commutative with exact rows and columns.
\[
    \xymatrix@C=1pc{
    {} & 1 \ar[d] & 1 \ar[d] & 1 \ar[d] & {} \\
    1 \ar[r] & H(r) \ar[r] \ar[d] & K \ar[r] \ar[d] & F \ar[r] \ar[d] & 1 \\
    1 \ar[r] & G_M^{(-)} \ar[r] \ar[d] & \widetilde{\Aut}(X) \ar[r] \ar[d]^{\tau} & \Aut(X) \ar[r] \ar[d] & 1 \\
    1 \ar[r] & \mathbb{Z} \ar[r] \ar[d] & \Gamma' \ar[r] \ar[d] & \Gamma'/\mathbb{Z} \ar[r] \ar[d] & 1 \\
    {} & 1 & 1 & 1 & {}
    }
    \hspace{1em}
    \xymatrix@C=1pc{
    {} & 1 \ar[d] & 1 \ar[d] & 1 \ar[d] & {} \\
    1 \ar[r] & Z(H(r)) \ar[r] \ar[d] & H(r) \ar[r] \ar[d] & H(r)/Z(H(r)) \ar[r] \ar[d] & 1 \\
    1 \ar[r] & Z(K) \ar[r] \ar[d] & K \ar[r] \ar[d] &K/Z(K) \ar[r] \ar[d] & 1 \\
    1 \ar[r] & F' \ar[r] \ar[d] & F \ar[r] \ar[d] &F'' \ar[r] \ar[d] & 1 \\
    {} & 1 & 1 & 1 & {}
    }
\]
Note that both \(F'\) and \(F''\) are finite by \cref{eq:smn_centre_one,eq:smn_centre_two,eq:smn_centre_three}.
Consequently, \(F\) is also a finite group.
Also, as before, the quotient \(\Gamma'/\mathbb{Z} \leq \Gamma/\mathbb{Z}\) is finite by \cref{lem:finite_quotient}.

\begin{theorem}\label{thm:aut_smn}
    Let \(X\) be an Inoue surface of type \(S_M^{(-)}\).
    Then the automorphism group \(\Aut(X)\) is finite.
    In particular, \(X\) does not have any wild automorphism.
    \qed
\end{theorem}

\begin{proof}[Proof of \cref{thm:wild_dim2}]
    Part~(1) is a consequence of \cref{prop:wild_kahler_dim2,prop:wild_nonkahler_dim2}.
    Part~(2) follows from \cite{cantat2020automorphisms}*{Lemma~6.7} and \cref{ex:smp} immediately.
\end{proof}

\begin{remark}
    Let $V$ be a compact complex manifold.
    Then the automorphism group $\Aut(V)$ is a complex Lie group whose Lie algebra is $\Gamma(V, \Theta_V)$.
    Here $\Theta_V$ is the holomorphic tangent bundle of $V$.

    By \cite{inoue1974surfaces}*{Propositions~2.ii), 3.ii) and 5},
    we have $\dim \Gamma(X, \Theta_X) = 0, 1, 0$
    for \(X\) an Inoue surface of type \(S_M\), \(S_M^{(+)}\), \(S_M^{(-)}\), respectively,
    with which our results are compatible.
\end{remark}

\section{Wild automorphisms in dimension three and four}\label{sec:wild_dim3}

In this section, we study wild automorphisms of compact K\"ahler manifolds, and
prove \cref{thm:main,thm:main_zero_entropy}
following the strategy of \cite{oguiso2022wild}.

\begin{proof}[Proof of \cref{thm:main}]

    The assertion (2) follows from the assertion (1), \cref{prop:entropy_tori} and \cite{oguiso2022wild}*{Lemma~6.1}. In what follows, we show the assertion (1).

    Suppose that \(X\) is neither a complex torus nor a weak Calabi--Yau threefold. Then the Kodaira dimension \(\kappa(X) < 0\)
    by Propositions \ref{prop:numerical_inv_of_wild} (2) and \ref{prop:qtorus}.
    Hence \(X\) is uniruled by \cite{brunella2006positivity}*{Corollary 1.2}. Moreover, by \cref{lem:mrc}, the MRC fibration \(f\colon X \longrightarrow Y\)
    is a \(\sigma\)-equivariant surjective smooth morphism with \(1 \leq \dim Y \leq 2\).
    Therefore, \(Y\) is a complex torus of dimension \(1\) or \(2\),
    and every fibre \(F\) of \(f\) is a smooth rational variety of dimension \(2\) or \(1\), respectively.
    Here we use the fact that a rationally connected variety of dimension at most two is a rational variety.

    We claim that \(Y = \Alb(X)\). In fact, since \(Y = \Alb(Y)\),
    there is a morphism \(g\colon \Alb(X) \longrightarrow Y\) such that \(g\circ \alb_X = f\)
    by the universal property of the Albanese map \(\alb_X\).
    Since \(f\) is surjective with connected fibres, so is \(g\).
    Let \(X_y\) (\(y\in Y\)) be any closed fibre of \(f\).
    Then \(\alb_X(X_y)\) is a point for each \(X_y\),
    because \(X_y\) is rationally connected, and any complex torus contains no rational curve.
    Thus, \(g\) is a finite morphism (with connected fibres),
    so it is an isomorphism by the normality of \(\Alb(X)\) and \(Y\).
    This proves the claim.

    If the irregularity \(q(X) = 1\), then \(Y\) is an elliptic curve, and every closed fibre \(X_y\) over \(y \in Y\) is a smooth rational surface. This implies that \(X\) is projective (cf.~\cite{prokhorov2021equivarianta}*{16.3.1~Proposition}),
    which contradicts \cite{oguiso2022wild}*{Proposition~5.2}.

    If the irregularity \(q(X) = 2\), then \(Y\) is a complex torus of dimension two,
    and every closed fibre \(X_y\) over \(y \in Y\) is a smooth rational curve: \(X_y \cong \PP^1\). Consider the \(\sigma\)-equivariant fibration \(f\colon X \longrightarrow Y\) with the induced wild automorphism \(\sigma_Y\) on \(Y\).
    Note that \(f\) is a locally trivial projective morphism.
    By \cref{thm_rrz7.2}, we can write the wild automorphism \(\sigma_Y = T_b\circ \alpha\) on \(Y\)
    where \(T_b\) is a translation on the complex torus \(Y\)
    and \(\alpha\colon Y \longrightarrow Y\) is a group automorphism
    such that the endomorphism \(\beta = \alpha - \id_Y\) is nilpotent.

    Since the anti-canonical divisor \(-K_{X_y}\) of a general fibre \(X_y \simeq \mathbb{P}^1\) is ample,
    the automorphism \(\sigma_Y = T_b\circ \alpha\) cannot be a translation by \cref{lem_oz2.11}.
    In particular, \(\beta \neq 0\).
    % If \(\beta = 0\), then \(\sigma_Y = T_b\) and it is a translation.
    % By \cref{lem_oz2.11}, the anti-canonical divisor \(-K_{X_y}\) of a general fibre \(X_y\),
    % is not ample, which contradicts that \(X_y \cong \PP^1\) has ample anti-canonical divisor.
    % Thus we may assume that \(\beta \neq 0\).
    Take \(B\) to be a connected component of \(\Ker \beta\)
    Then \(B\) being an one-dimensional subtorus of \(Y\), is an elliptic curve.
    % Let \(B\) be a connected component of \(\Ker \beta\).
    % Then \(B\) is a non-trivial proper subtorus of the \(2\)-dimensional complex torus \(Y\),
    % and hence an elliptic curve.
    Thus, \(\sigma_Y\) permutes the cosets of \(E\coloneqq Y/B\),
    an elliptic curve.
    Hence the quotient map \(Y \longrightarrow E\) is \(\sigma_Y\)-equivariant.
    Since the action of \(\sigma_Y\) on \(Y\) is wild,
    so is the induced action \(\sigma_E\) on \(E\) (cf.~\cref{lem:ppty_of_wild2} (2)).
    Hence \(\sigma_E\) is a translation of infinite order as \(\Aut(E)\) is an algebraic group (see \cref{prop:lie_gp_action_contains_wild}).

    Consider the $\sigma$-equivariant fibration $g\colon X \to E$, which is the composition $X \to Y \to E = Y/B$.
    This is a smooth fibration by \cref{lem:ppty_of_wild2}~(1).
    Note that every fibre of $X\to Y$ is isomorphic to $\PP^1$.
    Thus each fibre $X_e$ of $g$ over $e \in E$ is a relatively minimal ruled surface over the elliptic curve $Y_e$
    (which is the fibre of $Y\to E$ over $e \in E$).

    For the rest,
    the identical proof of \cite{oguiso2022wild}*{Proposition~4.1}
    still works here and shows that the case $q(X) = 2$ cannot happen,
    with their Lemma~2.5 and Theorem~3.1 replaced by our \cref{lem:ppty_of_wild2,prop:wild_kahler_dim2}, respectively.

    Therefore, \(X\) is either a complex torus or a weak Calabi--Yau threefold. This proves the assertion (1).
\end{proof}

Recall that a weak Calabi--Yau manifold \(V\) is a complex projective manifold with torsion canonical divisor and finite fundamental group.
In particular, \(H^1(\OO_V) = 0\) and \(\Pic^0(V)\) is trivial.
So on \(V\) the Picard group coincides with the N\'eron--Severi group, which is a finitely generated abelian group.

Now we consider a weak Calabi--Yau threefold \(X\).
By a result of Miyaoka (\cite{miyaoka1987chern}*{Theorem~6.6}), we have \(c_2(X) \cdot D \geq 0\) for each nef Cartier divisor \(D\) on \(X\).
Moreover, by \cite{kobayashi2014differential}*{Corollary~\Romannum{4}.4.15}, \(c_2(X) \neq 0\),
and thus, \(c_2(X) \cdot H > 0\) for every ample Cartier divisor \(H\) on \(X\).

\begin{proof}[Proof of \cref{prop:CY}]
    (1)
    Assume that
    \(c_2(X)\cdot D > 0\) for every non-torsion nef Cartier divisor \(D\) on \(X\).
    Then it is well-known that the automorphism group \(\Aut(X)\) is finite
    (see e.g., \cite{wilson1994minimal}*{Lemma~3.1}).
    Thus, \(X\) admits no wild automorphism in this case.

    (2)
    Assume that there is a non-torsion semi-ample Cartier divisor \(D\) such that \(c_2(X) \cdot D = 0\).
    Then some multiple of \(D\) induces a \(c_2\)-contraction
    (in the sense that \(c_2(X) \cdot D = 0\); cf.~\cite{oguiso2001calabi}*{Page~45}).
    Consider the maximal \(c_2\)-contraction \(\phi\colon X \to Y\) on \(X\)
    (see \cite{oguiso2001calabi}*{Lemma-Definition~4.1}).
    Then \(\phi\) is not an isomorphism and \(1\leq \dim Y \leq 3\).

    Suppose that \(X\) admits a wild automorphism \(\sigma\).
    Then by the uniqueness of the maximal \(c_2\)-contraction (\cite{oguiso2001calabi}*{Lemma-Definition~4.1}),
    \(\sigma\) descends to an automorphism on \(Y\) which is wild by \cref{lem:ppty_of_wild2}~(2).
    If \(\dim Y = 3\), then \(\phi\) is birational and hence an isomorphism by \cref{lem:ppty_of_wild2}~(3), a contradiction.
    If \(\dim Y \leq 2\), then by \cref{thm:wild_dim2}, \(Y\) is a complex torus.
    On the other hand, since \(\pi_0(X)\) is finite, \(H^1(\OO_X) = 0\), which implies \(H^1(\OO_Y) = 0\).
    This is a contradiction.
    % then it is well-known that \(Y\) is rational, again a contraction by \cref{prop:numerical_inv_of_wild}~(1).
    Therefore, \(X\) admits no wild automorphism.
\end{proof}

\begin{remark}
    \leavevmode
    \begin{enumerate}[wide=0pt, leftmargin=*]
        \item A similar argument to the proof of \cref{prop:CY} was also used in the proof of \cite{kirson2010wild}*{Theorem~4.7}.
              % Recall that the finiteness of $c_2$-contractions is obtained only after modulo the action of the automorphism group; see \cite{oguiso2001calabi}*{Theorem~(0.4)}.
              % In view of this, the proof of \cite{kirson2010wild}*{Theorem~4.7}, which implies \cref{prop:CY}, seems incomplete.
        \item We refer to \cite{oguiso2001calabi} for classification results and concrete examples of Calabi--Yau threefolds satisfying assumption (2) in \cref{prop:CY}.
    \end{enumerate}
\end{remark}

\begin{proof}[Proof of \cref{thm:main_zero_entropy}]
    Claim~(1) follows from \cref{thm:wild_dim2,thm:main}.
    For Claim~(2), first note that by \cref{prop:numerical_inv_of_wild}~(2), \(K_X \sim_{\QQ} 0\);
    moreover, we may assume that \(X\) coincides with its Beauville--Bogomolov minimal split cover
    and that each factor \(X_i\) of \(X\) is stable under \(\sigma\) with \(\sigma_{X_i}\) wild.
    By \cref{prop:numerical_inv_of_wild} (1),
    \(\chi(\mathcal{O}_{X_i}) = 0\),
    so \(X_i\) is either a complex torus or a Calabi--Yau threefold in the strict sense.
    Note that \(\sigma\) has zero entropy if and only if \(\sigma_{X_i}\) has zero entropy
    for every factor \(X_i\).
    Now Claim~\((2)\) follows from \cref{prop:entropy_tori} and \cref{thm:main}.

    Let us finally show Claim \((3)\).
    Note that \(X\) is smooth.
    By assumption and \cref{prop_oz2.15}, we may assume \(q(X) = 0\).
    It is equivalent to show the first dynamical degree \(d_1(\sigma) = 1\).
    Now the argument of \cite{oguiso2022wild}*{Lemma~6.3} is still available.
    In fact, the proof of \cite{oguiso2022wild}*{Lemma~6.3} involves Hodge theory
    which holds for compact K\"{a}hler manifolds (see \cite{voisin2002hodge}*{Chapter~\Romannum{2}, \S~6}).
\end{proof}

% References

\end{document}